\documentclass[12pt]{amsart}%
\usepackage[centertags]{amsmath}
\usepackage{amsmath}
\usepackage{amsthm}
\usepackage{graphicx}%
\usepackage{amsfonts}%
\usepackage{amssymb}
\usepackage{color}
\newtheorem{theorem}{Theorem}[section]

\newtheorem{lemma}[theorem]{Lemma}
\newtheorem{proposition}[theorem]{Proposition}
\theoremstyle{definition}
\newtheorem{definition}[theorem]{Definition}
\theoremstyle{remark}
\newtheorem{remark}[theorem]{Remark}

\numberwithin{equation}{section}
\newcommand{\ds}{\displaystyle}
\newcommand{\R}{\mathbb R}

\newcommand{\N}{\mathbb N}
\newcommand{\Z}{\mathbb Z}
\newcommand{\norm}[1]{\left\Vert#1\right\Vert}

\newcommand{\set}[2]{\left\{ #1 \,;\, #2 \right\}}

\begin{document}
\title{Completeness of sums of subspaces of bounded functions and applications}
\author{Jo${\rm \ddot e}$l BLOT{}$^1$ and
Philippe CIEUTAT{}$^2$}
%
%
\setcounter{footnote}{-1}
\renewcommand{\thefootnote}{\alph{footnote}}
\footnote{ {}$^1$ Laboratoire SAMM EA4543, Paris 1 Panth\'{e}on-Sorbonne, centre P.M.F., 90 rue de Tolbiac, 75634 Paris cedex 13,
France. E-mail address: blot@univ-paris1.fr}
\setcounter{footnote}{-1}
\renewcommand{\thefootnote}{\alph{footnote}}
\footnote{ {}$^2$ Laboratoire de Math\'ematiques de Versailles, Universit\'e Versailles-Saint-Quentin-en-Yvelines, 45 avenue des \'Etats-Unis, 78035 Versailles cedex,
France. E-mail address: philippe.cieutat@uvsq.fr}

\begin{abstract}
We give a new proof of  a characterization of the closeness of the range of a continuous linear  operator and of the closeness of the sum of two closed vector subspaces of a Banach space. Then we state  sufficient conditions for  the closeness of the sum of two closed subspaces of the Banach space of bounded functions and apply this result on various pseudo almost periodic spaces and pseudo almost automorphic spaces. 
\end{abstract}

%
\maketitle

\noindent
{\bf 2010 Mathematic Subject Classification:}  42A75, 42A99, 47A05.

\vskip10mm
%

\section{Introduction}
\label{vi}

The origin of this work is the important work of Zheng and Ding  \cite{Zhe} on the completeness of the space of weighted pseudo almost automorphic functions. Here we identify two abstract general tools which we used to give a new proof and to generalize  the result of Zheng and Ding. Then we apply these general tools to state the completeness of various spaces of pseudo almost periodic functions or pseudo almost automorphic functions.

\vskip1mm
The first general tool is a characterization of the closeness of the sum of two closed vector subspaces of a Banach space. Such a characterization is established from a characterization of the closeness of the image of a continuous linear operator. The second general tool is a radial retraction, a notion which is associated to the names of Dunkl and Williams \cite{Du-Wi}.
\vskip1mm

We describe the contents of this work. 
In Section \ref{viii} we give a new proof of a theorem on the characterization of the closeness of the image of a continuous linear operator, we prove that this result of characterization is equivalent to Inverse Mapping Theorem of Banach, and we extend this characterization to  unbounded closed linear operators. In Section \ref{ix} we provide a characterization of the closeness of the sum of two closed vector subspaces of a Banach space.
In Section \ref{i} we provide  sufficient conditions to ensure the  closeness of the sum of two closed vector subspaces of the Banach space of bounded functions under the supremum norm by using a result on a radial projection. In Section \ref{ii} we apply these general tools to state the completeness of 
the spaces of $\mu$-pseudo almost periodic and $\mu$-pseudo almost automorphic functions without restrictive hypothesis on the measure $\mu$.
In Section \ref{v} we study the completeness of the spaces of square-mean  $\mu$-pseudo almost periodic and  square-mean $\mu$-pseudo almost automorphic processes which are used in stochastic  evolution equations.
We states similar results in Section \ref{iv} for the space of $\mu$-pseudo almost periodic functions defined on the half-line and in Section \ref{iii} for the spaces of weighted pseudo almost periodic sequences and weighted pseudo almost automorphic sequences.

\section{Closed ranges}
\label{viii}

We provide a characterization of the closeness of the range of a linear operators. 
\vskip2mm

\begin{theorem}\label{th31}
Let $E$ and $F$ be two Banach spaces and $L : E \rightarrow F$ be a continuous linear mapping. The two following assertions are equivalent.
\vskip2mm
{\bf i)} ${\rm Im} \, L$ is closed into $F$.
\vskip2mm
{\bf ii)} $\exists c >0$, $\forall y \in Im L$, $\exists x \in E$ such that $L(x) = y$ and $\Vert x \Vert \leq c  \, \Vert y \Vert$.
\end{theorem}
%
\vskip2mm
%
\begin{remark}\label{rem32}
A proof of $ii) \Longrightarrow i)$ is given in \cite{BaBl}, Lemma 3.4, and the proof of $i) \Longrightarrow ii)$ is given in \cite{DS}, Lemma 1 in Section 6 of Chapter IV. 
Following our own knowledge, this result was first established by 
Ganesa Moorthy and Johnson in \cite{GMJ}. We give now another proof of 
Theorem \ref{th31}. 
The proof given below is very different and shorter than the two previously cited.  
\end{remark}

\vskip    2mm \noindent\textit{Proof of Theorem \ref{th31}.}
We consider the following equivalence relation on $E$: 
$$x \sim y \Longleftrightarrow L(x) = L(y)\Longleftrightarrow x-y \in {\rm Ker} \, L.$$
We consider the quotient space $E^{\bullet} := E / \sim$. An element $x^{\bullet}$ of 
$E^{\bullet}$ is $x^{\bullet} = x + {\rm Ker} \, L$ when $x \in E$. We know that $E^{\bullet}$ is a vector space. 
Since $L$ is continuous, ${\rm Ker} \, L$ is closed into $E$, that permits to $\Vert x^{\bullet} \Vert^{\bullet} := \inf \{ \Vert u \Vert : u \in x + {\rm Ker} \, L \}$ to be a norm on $E^{\bullet}$.  Since $E$ is complete, $(E^{\bullet}, \Vert \cdot \Vert^{\bullet})$ is a Banach space, cf. \cite{S} p. 299. We define $L^{\bullet} : E^{\bullet} \rightarrow F$ by setting $L^{\bullet}(x^{\bullet}) := L(x)$ when $x \in E$. $L^{\bullet}$ is linear and continuous. $L^{\bullet}$ is clearly injective. We define the following abridgement of $L^{\bullet}$: the mapping $K : E^{\bullet} \rightarrow {\rm Im} \,L$ defined by $K(x^{\bullet}) := L^{\bullet}(x^{\bullet})$. $K$ remains linear, continuous, injective and moreover $K$ is surjective. And so we can define the inverse of $K$, $K^{-1} : {\rm Im} \,L \rightarrow E^{\bullet}$, which is linear. 
\vskip2mm
{\bf i) $\Longrightarrow$ ii)}. Since ${\rm Im} \,L$ is closed in the Banach space $F$, it is a Banach space. Since $E^{\bullet}$ is a Banach space, we can use the Inverse Mapping Theorem of Banach and obtain that $K^{-1}$ is continuous. Therefore we have
$$\exists c_0 >0, \forall y \in {\rm Im} \,L, \qquad \Vert K^{-1}(y) \Vert^{\bullet} \leq c_0 \, \Vert y \Vert.$$
Let $y$ an element of ${\rm Im} \, L$.  We arbitrarily fix $x_0 \in K^{-1}(y)$. When $y \neq 0$, we have $\Vert K^{-1}(y) \Vert^{\bullet} < 2c_0 \, \Vert y \Vert$ which means that $\inf \{ \Vert u \Vert : u \in x_0 + {\rm Ker} \, L \} < 2c_0 \, \Vert y 
\Vert$, so there exists $x \in x_0 + {\rm Ker} \, L$  such that $\Vert x \Vert \leq 2c_0 \, \Vert y \Vert$; and moreover $L(x)=y$. 
When $y = 0$ we take $x = 0$ and we obtain the same inequality. Therefore ii) is proven with $c=2c_0$.
\vskip2mm
{\bf ii) $\Longrightarrow$ i)}. Let $y \in {\rm Im} \,L$. From ii) we know that there exists $x \in L^{-1}( \{ y \}) = K^{-1}(y)$ such that $\Vert x \Vert \leq c \, \Vert y \Vert$ that implies: $\Vert x^{\bullet} \Vert^{\bullet} \leq \Vert x \Vert \leq c \, \Vert y \Vert$. And so we have proven
$$\exists c >0, \forall y \in {\rm Im} \, L, \qquad \Vert K^{-1}(y) \Vert^{\bullet} \leq c \, \Vert y \Vert.$$
Since $K^{-1}$ is linear, this last assertion means that $K^{-1}$ is continuous. Let $(y_k)_k$ be a Cauchy sequence in ${\rm Im} \,L$. Since $K^{-1}$ is Lipschitzian, $(K^{-1}(y_k))_k$ is also a Cauchy sequence in $E^{\bullet}$. Since $E^{\bullet}$ is complete, there exists $z^{\bullet} \in E^{\bullet}$ such $(K^{-1}(y_k))_k$ converges toward $z^{\bullet}$. Since $K$ is continuous $(y_k)_k = (K(K^{-1}(y_k))_k$ converges toward $K(z^{\bullet})$ into ${\rm Im} \,L$. And so ${\rm Im} \,L$ is complete and therefore it is closed.
\hfill$\Box$

\begin{remark}\label{rem33}
The previous proof shows that Theorem \ref{th31} is a consequence of Inverse Mapping Theorem of Banach. In fact, 
Theorem \ref{th31} is equivalent to Inverse Mapping Theorem of Banach. If we assume that Theorem \ref{th31} is true, and if we consider 
\newline
$L \in {\mathcal L}(E,F)$, where $E$ and $F$ are Banach spaces and $L$ is bijective, since $L$ is surjective, we have ${\rm Im} \,L = F$ is closed in $F$ and consequently, using Theorem \ref{th31}, we obtain: $\exists c >0$, $\forall y \in F$, $\Vert L^{-1}(y) \Vert \leq c \, \Vert y \Vert$,
which means the continuity of the inverse $L^{-1}$.
\end{remark}

\vskip1mm

We can extend Theorem \ref{th31} to unbounded closed linear operators.

\begin{theorem}\label{th34}
Let $E$ and $F$ be Banach spaces and $T : {\mathfrak D}(T) \subset E \rightarrow F$ be a unbounded linear operator. We assume that $T$ is closed. Then the two following assertions are equivalent.
\vskip2mm
{\bf i)} ${\rm Im} \,T$ is closed into $F$.
\vskip2mm
{\bf ii)} $\exists c >0$, $\forall y \in {\rm Im} \, T$, $\exists x \in {\mathfrak D}(T)$ such that $T(x) = y$ and $\Vert x \Vert_E \leq c  \,\Vert y \Vert_F$.
\end{theorem}

\begin{proof}
First recall that $\Vert x \Vert_{\mathcal G} :=   \Vert x \Vert_E +\Vert T(x) \Vert_F$ defines a norm on ${\mathfrak D}(T)$ and the closeness of $T$ implies that $({\mathfrak D}(T), \Vert \cdot \Vert_{\mathcal G})$ is complete (cf. \cite{Ka} p.164). 
We define the following abridgement of $T$: the mapping $L : {\mathfrak D}(T) \rightarrow F$ defined by $L(x) := T(x)$ for $x\in {\mathfrak D}(T)$.
Since, $\forall x \in {\mathfrak D}(T)$, $\Vert L(x) \Vert_F=\Vert T(x) \Vert_F \leq   \Vert x \Vert_{\mathcal G}$, we have $L \in {\mathfrak L}( ({\mathfrak D}(T), \Vert \cdot \Vert_{\mathcal G}), (F, \Vert \cdot \Vert_F))$.
We introduce the following new condition.
\begin{equation}
\label{eq6}
\exists d >0, \forall y \in {\rm Im} \, L, \exists x \in {\mathfrak D}(T) \text{ such that }L(x) = y \text{ and } \Vert x \Vert_{\mathcal G} \leq d \, \Vert y \Vert_F .
\end{equation}
The equivalence i) $\Longleftrightarrow$ \eqref{eq6}  is a consequence of 
Theorem \ref{th31} applied on the
continuous linear mapping
$L \in {\mathfrak L}( ({\mathfrak D}(T), \Vert \cdot \Vert_{\mathcal G}), (F, \Vert \cdot \Vert_F))$  by remarking that ${\rm Im} \,L = {\rm Im} \,T$.
\vskip2mm
Let us  prove the equivalence \eqref{eq6} $\Longleftrightarrow$ ii).
For the implication \eqref{eq6} $\Longrightarrow ii)$ it suffices to take $c= d$ and to note that $\Vert x \Vert_E \leq \Vert x \Vert_{\mathcal G}$.
Now we prove the reciprocal implication.
Let $y \in {\rm Im} \, T$. Then, from ii), $\exists x \in {\mathfrak D}(T)$ such that $T(x) = y$ and $\Vert x \Vert_E \leq c  \,\Vert y \Vert_F$ which imply 
$\Vert x \Vert_{\mathcal G} =   \Vert x \Vert_E +\Vert T(x) \Vert_F =   \Vert x \Vert_E + \Vert y \Vert_F \leq  (c+1) \,  \Vert y \Vert_F$.
Setting $d := c+1$ we obtain $\Vert x \Vert_{\mathcal G}  \leq  d \,  \Vert y \Vert_F$.
And so we have proven the following assertion ii) $\Longrightarrow$ \eqref{eq6}.
\end{proof}

\section{General closed sums}
\label{ix}

In this section we provide necessary and sufficient conditions to have the closeness of the sum of two closed vector sunspaces in a Banach space.

\begin{theorem}\label{th41}
Let $E$ be a Banach space, and $M$ and $N$ be two closed vector subspaces of $E$. Then the three following conditions are equivalent.
\vskip2mm
{\bf i)} $M+N$ is closed into $E$.
\vskip2mm
{\bf ii)} $\exists c >0$, $\forall z \in M+N$, $\exists (x,y) \in M \times N$ such that $x+y = z$  and $\Vert x \Vert \leq c \,\Vert z \Vert$.
\end{theorem}

\begin{proof}
Let us consider the following assertion
\begin{equation}
\label{eq20}
\exists d >0, \forall z \in M+N, \exists (x,y) \in M \times N \text{ s.t. } x+y = z   , \Vert x \Vert \leq d \, \Vert z \Vert \text{ , }  \Vert y \Vert \leq d \, \Vert z \Vert.
\end{equation}
The equivalence $i) \Longleftrightarrow$ \eqref{eq20} is a consequence of 
Theorem \ref{th31}
by using  the continuous linear  mapping $L : M \times N \rightarrow E$ defined by  $L(x,y) := x+y$
and endowing the product $M \times N$ with the norm $\Vert (x,y) \Vert  = \max\{\Vert x \Vert ,  \Vert y \Vert\}$. 
The implication \eqref{eq20} $\Longrightarrow ii)$ is obvious with $c=d$ and for the reciprocal implication, note that if $\norm{x}\leq c \, \norm{z}$, we have
$$\Vert y \Vert - \Vert z \Vert \leq \Vert y - z \Vert = \Vert x \Vert \leq c \, \Vert z \Vert \Longrightarrow \Vert y \Vert \leq (c+1) \, \Vert z \Vert,$$
then the implication ii) $\Longrightarrow$ \eqref{eq20} is proved with $d=c+1$.
\end{proof}

\begin{remark}\label{rem42}
There exists another proof of Theorem \ref{th41} due to  Zheng and  Ding in \cite{Zhe2}.
The proof of i) $\Longrightarrow$ \eqref{eq20} is given in \cite{Br} (Theorem 2.10, p. 37).
A result as Theorem \ref{th41} is established by Kober \cite{Ko} under the additional assumption $M \cap N = \{ 0 \}$.
\end{remark}

\section{Closeness of the sum two vector subspaces of a Banach space of bounded functions}
\label{i}

In this section $U$ is a set and $X$ is a Banach space. $B(U,X)$ denotes the Banach space of all bounded functions from $U$ into $X$, equipped with the norm $\ds\norm{f}_{\infty} = \sup_{u\in U} \norm{f(u)}$ for $f \in B(U,X)$.
Here we establish  sufficient conditions to have the closeness of the sum of two closed subspaces of the Banach space of bounded functions   using a result on the radial retraction.

\begin{theorem} 
\label{Tw2} 
Let $\mathcal{F}_1$ and $\mathcal{F}_2$ be two closed vector subspaces of $B(U,X)$ satisfying the following hypothesis:
\vskip 2mm 
{\bf i)} The map $\Phi : X\to X$ is  Lipschitz and  $f\in\mathcal{F}_1$, then $\Phi\circ f\in\mathcal{F}_1$.
\vskip 2mm 
{\bf ii)} $g\in\mathcal{F}_1$ and $h\in\mathcal{F}_2$ such that $\forall u\in U$, $\norm{g(u)}\leq\norm{h(u)}$, then $g\in\mathcal{F}_2$.
\vskip         2mm
Then $\mathcal{F}_1 + \mathcal{F}_2$ is a Banach space endowed with the norm $\norm{\cdot}_{\infty}$. 
\end{theorem}

For the proof of Theorem \ref{Tw2}, we use the radial retraction $P_R$ on the ball of radius $R$ and with the center at the origin.

\begin{lemma} 
\label{Tw1} 
Let $R\geq 0$. We define the function $P_R : X \to X$ by setting
\begin{equation}
\label{Eqw1}
P_R(x)=
\left\{
\begin{array}
[c]{l}
x \quad\qquad\text {if }\qquad \norm{x}\leq R \\
 \\
\frac{R}{\norm{x}} x \qquad\text {if }\qquad \norm{x}> R .
\end{array}
\right.
\end{equation}
Then the following assertions hold:
\vskip             2mm 
{\bf i)} $P_R$ is $2$-Lipschitz: $\forall x_1$, $x_2\in X$, $\norm{P_R(x_1)-P_R(x_2)}\leq 2\norm{x_1-x_2}$.
\vskip             2mm
{\bf ii)} For all $x\in X$, $\norm{P_R(x)}\leq R$.
\end{lemma}

\begin{proof}
{\bf i)} If
\begin{equation}
\label{Eq10}
\norm{x_1}\leq R \leq \norm{x_2} ,
\end{equation}
then
$$\norm{P_R(x_1) - P_R(x_2)} = \norm{x_1 - \frac{R}{\norm{x_2}}x_2} \leq \norm{x_1 -x_2} + \norm{x_2 - \frac{R}{\norm{x_2}}x_2} ,$$
from
$\ds\norm{x_2 - \frac{R}{\norm{x_2}}x_2} = \left\vert \norm{x_2}- R \right\vert$ and \eqref{Eq10} we obtain
$$\norm{x_2 - \frac{R}{\norm{x_2}}x_2} \leq \norm{x_2}-\norm{x_1} \leq \norm{x_1-x_2} ,$$
therefore
\begin{equation}
\label{Eq11}
\norm{P_R(x_1)-P_R(x_2)} \leq 2\norm{x_1-x_2} .
\end{equation}
If $\norm{x_1}\leq R$ and $\norm{x_2}\leq R$, then we have \eqref{Eq11}, since $P_R(x_1)=x_1$ and $P_R(x_2)=x_2$.
Dunkl and Williams have proved in \cite{Du-Wi} that for all nonzero $x_1$ and $x_2\in X$
\begin{equation}
\label{Eq12}
\norm{\frac{x_1}{\norm{x_1}} - \frac{x_2}{\norm{x_2}}} \leq \frac{4}{\norm{x_1}+\norm{x_2}} \norm{x_1-x_2} .
\end{equation}
From 
$$\norm{P_R(x_1) - P_R(x_2)} = \norm{\frac{R}{\norm{x_1}}x_1 - \frac{R}{\norm{x_2}}x_2}$$
and \eqref{Eq12} it can be seen that \eqref{Eq11} is fulfilled if $\norm{x_1}\geq R$ and $\norm{x_2}\geq R$. This ends the proof.
\vskip 2 mm 
{\bf ii)} From  \eqref{Eqw1} we deduce ii).
\end{proof}

\vskip    2mm \noindent\textit{Proof of Theorem \ref{Tw2}.}
We use Theorem \ref{th41} by setting $E=B(U,X)$, $M=\mathcal{F}_1$ and $N=\mathcal{F}_2$.
Let us prove that 
\begin{equation}
\label{Eq4}
\forall f\in \mathcal{F}_1+\mathcal{F}_2 , \quad \exists g^*\in \mathcal{F}_1 , \quad \exists h^*\in \mathcal{F}_2 \text{ such that } f=g^*+h^* \text{ and } \norm{g^*}_{\infty}\leq\norm{f}_{\infty}  .
\end{equation}
Let $f\in \mathcal{F}_1+\mathcal{F}_2$. Then there exist $g\in \mathcal{F}_1$ and $h\in \mathcal{F}_2$ such that $f=g+h$. 
Define $g^*$ by setting $g^* = P_R\circ g$, with $R:=\norm{f}_{\infty}$ and $P_R$ the radial retraction of Lemma \ref{Tw1}. Define $h^*$ by setting $h^* = f - g^*$. Then we have $f=g^*+h^*$.
\vskip 1 mm
Since $P_R : X\to X$ is  Lipschitz and $g\in \mathcal{F}_1$, from Hypothesis i), we obtain
$P_R\circ g=g^*\in \mathcal{F}_1$. 
Inequality $\norm{g^*}_{\infty}\leq R=\norm{f}_{\infty}$ results of ii) of Lemma \ref{Tw1}. 
\vskip 1 mm
Since $\norm{f(u)}\leq R:=\norm{f}_{\infty}$ for all $u\in U$ and from the definition of  $P_R$, we deduce that $P_R\circ f = f$; therefore
$\norm{h^*(u)} = \norm{f(u) - g^*(u)} = \norm{P_R(f(u)) - P_R(g(u))}$. The radial retraction $P_R$ being $2$-Lipschitz, we obtain that
$$\forall u\in U,\qquad\norm{h^*(u)} =\norm{P_R(f(u)) - P_R(g(u))} \leq 2\norm{f(u)-g(u)} = 2\norm{h(u)} .$$
From $h\in\mathcal{F}_2$ and Hypothesis ii), we deduce that $h^* \in \mathcal{F}_2$; so \eqref{Eq4} is fulfilled, which is assertion ii) of Theorem \ref{th41} with $c=1$. 
\vskip 1 mm
By using Theorem \ref{th41} with $E=B(U,X)$, $M=\mathcal{F}_1$,  $N=\mathcal{F}_2$ and $c=1$,
we obtain the closeness of the vector subspace $\mathcal{F}_1+\mathcal{F}_2$ in the Banach space $(B(U,X),\norm{\cdot}_{\infty})$. 
\hfill$\Box$

\begin{remark} 
Under the hypotheses of Theorem \ref{Tw2}, if in addition we assume that $\mathcal{F}_1 \cap \mathcal{F}_2 =\{0\}$, the proof of this theorem is more simple. In this case, the decomposition (which becomes unique) $f = g+h$ with $g \in \mathcal{F}_1$ and $h \in \mathcal{F}_2$, we obtain $\Vert g \Vert_{\infty} \leq \Vert f \Vert_{\infty}$ from \eqref{Eq4}. This fact and the closeness of $\mathcal{F}_1$ and $\mathcal{F}_2$ imply the completeness of the sum.
\end{remark}

\section{Spaces of $\mu$-pseudo almost periodic and automorphic functions}
\label{ii}

In this section $X$ is a Banach space.

\subsection{Almost periodic case}

The pseudo almost periodic functions were introduced in the literature in the early nineties by Zhang \cite{Zha1, Zha2}, as a natural generalization of the classical almost periodic functions in the sense of Bohr. 
The notion of $\mu$-pseudo almost periodicity initiated by Blot, Cieutat and Ezzinbi \cite {27} is a generalization of the weighted pseudo almost periodicity due to Diagana \cite{Dia1, Dia2} which generalizes the one of pseudo almost periodicity.

\begin{definition}
\label{Def4}
\cite{Co}. $f\in C(\R,X)$ 
(continuous)
 is said to be
 \textit{almost periodic} (in the Bohr sense) if for all $\varepsilon>0$, there exists
${\ell}>0$, such that for all $\alpha\in\mathbb{R}$, there exists $\tau
\in\lbrack\alpha,\alpha+{\ell}]$ with
\[
\sup_{t\in\mathbb{R}}\left\|  f(t+\tau)-f(t)\right\|  <\varepsilon.
\]
We denote the space of all such functions by $AP(\mathbb{R},X)$.
\end{definition}

\vskip 2mm

We denote by $\mathcal{B}$ the Lebesgue $\sigma$-field of $\mathbb{R}$ and by
$\mathcal{M}$ the set of all positive measures $\mu$ on $\mathcal{B}$
satisfying $\mu(\mathbb{R})=+\infty$ and $\mu([a,b])<+\infty$, for all $a$,
$b\in\mathbb{R}$ ($a\leq b$).
$BC(\mathbb{R},X)$ denotes the
Banach space of all  continuous and bounded functions from $\mathbb{R}$ to $X$,
equipped with the norm
\[
\left\|  f\right\|  _{\infty}=\sup_{t\in\mathbb{R}}\left\|  f(t)\right\|  .
\]

\begin{definition}
\label{Def1}
\cite {26, 27}. Let $\mu\in\mathcal{M}$. $f\in BC(\R,X)$
is said to be \textit{$\mu$-ergodic} if 
\[
\underset{r\rightarrow+\infty}{\lim}\frac{1}{\mu([-r,r])}\int_{[-r,r]}
\left\Vert f(t)\right\Vert \,d\mu(t)=0 .
\]
We denote the space of all such functions by $\mathcal{E}(\mathbb{R}
,X,\mu)$.
\end{definition}

\begin{definition}
\label{Def5}
\cite {27}. Let $\mu\in\mathcal{M}$. $f\in BC(\R,X)$ is said to be \textit{$\mu$-pseudo almost periodic} if 
\[
f = g+\phi \quad \text{ where }g\in AP(\mathbb{R},X) , \quad \phi\in\mathcal{E}(\mathbb{R},X,\mu).
\]
We denote the space of all such functions by $PAP(\mathbb{R},X,\mu)$. 
\end{definition}

\begin{remark} 
In general the sum $PAP(\mathbb{R},X,\mu)=AP(\mathbb{R},X)+\mathcal{E}(\mathbb{R},X,\mu)$ is not direct (\cite{27}, Remark 2.26).
\end{remark}

\begin{theorem} 
\label{Tw3}
Let $\mu\in\mathcal{M}$. Then $(PAP(\mathbb{R},X,\mu),\left\Vert \cdot\right\Vert _{\infty})$ is a Banach space.
\end{theorem}

\begin{proof}
We use Theorem \ref{Tw2} by setting $U=\R$, $\mathcal{F}_1=AP(\R,X)$ and $\mathcal{F}_2=\mathcal{E}(\mathbb{R},X,\mu)$.
\vskip 2 mm
{\bf i)} $AP(\R,X)$ is a closed vector subspace of $BC(\R,X)$ (\cite{Co}, Theorem 6.1 and 6.2, pp. 138-139), then a closed vector subspace of $B(\R,X)$. Moreover if $\Phi\in C(X,X)$ and $f\in AP(\R,X)$ then $\Phi\circ f\in AP(\R,X)$ (\cite{AmPr}, VII, p. 5); so Hypothesis i) is satisfied.
\vskip 2 mm
{\bf ii)} $\mathcal{E}(\mathbb{R},X,\mu)$ endowed with the norm $\norm{\cdot}_{\infty}$ is a Banach space (\cite{27}, Proposition 2.11), then a closed vector subspace of $B(\R,X)$. By using Definition \ref{Def1}, we deduce that the pair $AP(\R,X)$ and $\mathcal{E}(\mathbb{R},X,\mu)$ satisfies Hypothesis ii).
\vskip 2 mm
By Theorem \ref{Tw2}, we obtain that $PAP(\mathbb{R},X,\mu)=AP(\mathbb{R},X) + \mathcal{E}(\mathbb{R},X,\mu)$  is a Banach space endowed with the norm $\norm{\cdot}_{\infty}$. 
\end{proof}

\begin{remark} 
With restrictive conditions on the measure $\mu$, this result is proved (\cite{27}, Corollary 2.31).
\end{remark}

\begin{remark} 
Theorem \ref{Tw3} generalizes Corollary 2.4 in \cite{Zhe} of Zheng and Ding from the space of weighted pseudo almost periodic functions to the space of $\mu$-pseudo almost periodic functions. 
\end{remark}

\subsection{Almost automorphic case}

The notion of $\mu$-pseudo almost automorphy  due to Blot, Cieutat and Ezzinbi  \cite {26} is a generalization of the weighted pseudo almost automorphy initiated by  Blot, Mophou, N'Gu$\acute{e}$r$\acute{e}$kata and Pennequin
\cite{Blo} which generalizes the one of pseudo almost automorphy introduced by 
Liang, N'Gu\'er\'ekata, Xiao and Zhang \cite{Lia}.

\begin{definition}
\label{Def6}
\cite{ngu01}. $f\in C(\R,X)$ is said to
be \textit{almost automorphic} if for every sequence of real numbers $(s_{n}^{\prime
})_{n}$ 
there exists a subsequence of $(s_{n}^{\prime})_{n}$ denoted by $(s_{n})_{n}$ such that
\[
g(t)=\underset{n\rightarrow\infty}{\lim}f(t+s_{n})\text{ exists for all
}t\text{ in }\mathbb{R}
\]
and
\[
\underset{n\rightarrow\infty}{\lim}g(t-s_{n})=f(t)\text{ for all }t\text{ in
}\mathbb{R}\text{.}
\]
We denote the space of all such functions by $AA(\mathbb{R},X)$.

\end{definition}

\begin{definition}
\label{Def7}
\cite {26}. Let $\mu\in\mathcal{M}$. $f\in BC(\R,X)$ is said to be \textit{$\mu$-pseudo almost automorphic} if 
\[
f = g+\phi \quad \text{ where }g\in AA(\mathbb{R},X) , \quad \phi\in\mathcal{E}(\mathbb{R},X,\mu).
\]
We denote the space of all such functions by $PAA(\mathbb{R},X,\mu)$. 
\end{definition}

\begin{remark} 
In general the sum $PAA(\mathbb{R},X,\mu)=AA(\mathbb{R},X)+\mathcal{E}(\mathbb{R},X,\mu)$ is not direct (\cite{26}, Remark 4.4). A sufficient condition for $PAA(\mathbb{R},X,\mu)=AA(\mathbb{R},X) \bigoplus \mathcal{E}(\mathbb{R},X,\mu)$
is that $PAA(\mathbb{R},X,\mu)$ be translation invariant.
\end{remark}

\begin{theorem} 
\label{Tw4}
Let $\mu\in\mathcal{M}$. Then $(PAA(\mathbb{R},X,\mu),\left\Vert \cdot\right\Vert _{\infty})$ is a Banach space.
\end{theorem}

\begin{proof}
We use Theorem \ref{Tw2} by setting $U=\R$, $\mathcal{F}_1=AA(\R,X)$ and $\mathcal{F}_2=\mathcal{E}(\mathbb{R},X,\mu)$.
\vskip 2 mm
{\bf i)} $AA(\R,X)$ is a closed vector subspace of $BC(\R,X)$ (\cite{ngu01}, Theorem 2.3, p. 11 and Theorem 2.10, p. 16), then a closed vector subspace of $B(\R,X)$. Moreover if $\Phi\in C(X,X)$ and $f\in AA(\R,X)$ then $\Phi\circ f\in AA(\R,X)$ (\cite{ngu01}, Theorem 2.5, p. 13); so Hypothesis i) is satisfied.
\vskip 2 mm
{\bf ii)} is similar to ii) in the proof of Theorem \ref{Tw3}.
\vskip 2 mm
By Theorem \ref{Tw2}, we deduce that $PAA(\mathbb{R},X,\mu)=AA(\mathbb{R},X)+\mathcal{E}(\mathbb{R},X,\mu)$  is a Banach space endowed with the norm $\norm{\cdot}_{\infty}$. 
\end{proof}

\begin{remark} 
In the particular case where $PAA(\mathbb{R},X,\mu)$ is translation invariant, this result is proved in (\cite{26}, Theorem 4.9).
\end{remark}

\begin{remark}
Theorem \ref{Tw4} generalizes Theorem 2.3 in \cite{Zhe} of Zheng and Ding from the space of weighted pseudo almost automorphic functions to the space of $\mu$-pseudo almost  automorphic functions.  
\end{remark}

\section{Spaces of square-mean $\mu$-pseudo almost periodic and automorphic processes}
\label{v}

Throughout this section, we denote by $H$ a Banach space, $(\Omega, \mathcal{A}, P)$ a probability space and $L^2(\Omega, P, H)$ the space of all $H$-valued random variables $x$ with a finite quadratic-mean:
$$E\norm{x}^2 =  \int_{\Omega} \norm{x}^2 \, dP<+\infty .$$
 When $x\in L^2(\Omega, P, H)$, we set $\ds\norm{x}_{L^2} = \left( E\norm{x}^2\right)^{\frac{1}{2}}$. Endowed with the norm $\norm{\cdot}_{L^2}$, $L^2(\Omega, P, H)$ is a Banach space.

\begin{definition}
\cite{Bez-Dia-2}.
Let $x : \R\to L^2(\Omega, P, H)$ be a stochastic process.
\vskip 2 mm
{\bf i)}
 $x$ is said to be \textit{stochastically continuous} if
$$\forall t\in\R , \qquad \lim_{s\to t} E\norm{x(s) - x(t)}^2 = 0 .$$
\vskip 2 mm
{\bf ii)}
$x$  is said to be \textit{stochastically bounded} if
$$\sup_{t\in\R} E\norm{x(t)}^2 <+\infty  .$$
\end{definition}

\begin{remark}
It is easy to verify that a stochastic process $x : \R\to L^2(\Omega, P, H)$ is stochastically continuous and bounded
if and only if
$x$ is a continuous and bounded map from $\R$ to the Banach space $X=L^2(\Omega, P, H)$ in the ordinary sense, i.e. $x\in B(\R,X)$.
\end{remark}

\vskip2mm

The square-mean almost periodic stochastic processes were introduced in the literature by Bezandry and Diagana \cite{Bez-Dia-2}. They established the existence and uniqueness of square-mean almost periodic mild solutions to some stochastic differential equations and some functional integro-differential stochastic evolution equations in \cite{Bez-Dia-2, Bez-Dia-1}. Then Fu and Liu generalize this one to the square-mean almost automorphy in \cite{Fu-Liu}. 
The notion of square-mean $\mu$-pseudo almost periodic and automorphic processes due to 
Diop, Ezzinbi and Mbaye \cite {Di-Ez-Mb-1, Di-Ez-Mb-2}  is a generalization of the pseudo almost automorphic stochastic processes  \cite{Ch-Li-2} and the weighted pseudo almost automorphy \cite{Ch-Li-1} due to Chen and Lin.

The definition of the square-mean $\mu$-pseudo almost automorphy given  in \cite {Di-Ez-Mb-2} is slightly different  from the one given by 
Bedouhene, Challali, Mellah, Raynaud de Fitte and Smaali  in  \cite{BeChMeRaSm}.
A continuous and bounded stochastic process $x : \R\to L^2(\Omega, P, H)$ is 
square-mean $\mu$-ergodic in \cite {Di-Ez-Mb-2} if 
$$\underset{r\rightarrow+\infty}{\lim}\frac{1}{\mu([-r,r])}\int_{[-r,r]}
E\norm{x(t)}^2 \,d\mu(t)=0 
$$
and in \cite{BeChMeRaSm}
 if 
$$\underset{r\rightarrow+\infty}{\lim}\frac{1}{\mu([-r,r])}\int_{[-r,r]}
\left(E\norm{x(t)}^2\right)^{\frac{1}{2}} \,d\mu(t)=0 .
$$
We will see  that these definitions are equivalent (cf. Proposition \ref{Prop1} and \ref{Prop2}).

\subsection{Almost periodic case}

\begin{definition}
\cite{Bez-Dia-2}. A continuous stochastic process $x : \R\to L^2(\Omega, P, H)$ is said to be
 \textit{square-mean almost periodic}  if for all $\varepsilon>0$, there exists
${\ell}>0$, such that for all $\alpha\in\mathbb{R}$, there exists $\tau
\in\lbrack\alpha,\alpha+{\ell}]$ with
\[
\sup_{t\in\mathbb{R}} E \norm{x(t+\tau)-x(t)}^2  <\varepsilon.
\]
\end{definition}

\begin{remark}
\label{Rq1}
It is easy to verify that a stochastic process $x : \R\to L^2(\Omega, P, H)$ is square-mean almost periodic
if and only if
$x$ is an almost periodic function from $\R$ to the Banach space $X=L^2(\Omega, P, H)$ in the Bohr sense, i.e. $x\in AP(\R,X)$ (cf. Definition \ref{Def4}).
\end{remark}

\vskip 2mm

As in Section \ref{ii}, we denote by $\mathcal{B}$ the Lebesgue $\sigma$-field of $\mathbb{R}$ and by
$\mathcal{M}$ the set of all positive measures $\mu$ on $\mathcal{B}$
satisfying $\mu(\mathbb{R})=+\infty$ and $\mu([a,b])<+\infty$, for all $a$,
$b\in\mathbb{R}$ ($a\leq b$).

\begin{definition} \cite{Di-Ez-Mb-2}
Let $\mu\in\mathcal{M}$. Let $x : \R\to L^2(\Omega, P, H)$ be a continuous and bounded stochastic process.
\vskip 2mm
{\bf i)} $x$ is said to be \textit{square-mean $\mu$-ergodic} if 
\[
\underset{r\rightarrow+\infty}{\lim}\frac{1}{\mu([-r,r])}\int_{[-r,r]}
E\norm{x(t)}^2 \,d\mu(t)=0 .
\]
\vskip 2mm
{\bf ii)} $x$ is said to be \textit{square-mean $\mu$-pseudo almost periodic} if 
\[
x = y+z
\]
where $y$ is square-mean almost periodic and $z$ square-mean $\mu$-ergodic.
\end{definition}

\begin{proposition}
\label{Prop1}
Let $\mu\in\mathcal{M}$. Let $x : \R\to L^2(\Omega, P, H)$ be a stochastic process.
\vskip 2mm
{\bf i)} $x$ is square-mean $\mu$-ergodic
if and only if
$x$ is a $\mu$-ergodic function from $\R$ to the Banach space $X=L^2(\Omega, P, H)$ in the sense of Definition \ref{Def1}, i.e. $x\in \mathcal{E}(\mathbb{R}
,X,\mu)$.
\vskip 2mm
{\bf ii)} $x$ is square-mean $\mu$-pseudo almost periodic
if and only if
$x$ is a $\mu$-pseudo almost periodic function from $\R$ to the Banach space $X=L^2(\Omega, P, H)$ in the sense of Definition \ref{Def5}, i.e. $x\in PAP(\mathbb{R}
,X,\mu)$.
 \end{proposition}

\begin{proof} 
Recall that the Banach space $X=L^2(\Omega, P, H)$ is endowed with the norm $\ds\norm{\cdot}_{L^2}$
 defined by $\ds\norm{x}_{L^2} = \left( E\norm{x}^2\right)^{\frac{1}{2}}$ for $x\in X$.
\vskip 2 mm
{\bf i)} Let $x : \R\to L^2(\Omega, P, H)$ be a stochastic process. It suffices to prove that $x$ is square-mean $\mu$-ergodic if and only if
\begin{equation}
\label{Eq9}
\underset{r\rightarrow+\infty}{\lim}\frac{1}{\mu([-r,r])}\int_{[-r,r]}
\left(E\norm{x(t)}^2\right)^{\frac{1}{2}} \,d\mu(t)=0 .
\end{equation}
First from Cauchy-Schwartz inequality, it follows
\begin{equation*}
\int_{[-r,r]} \norm{x(t)}_{L^2} \,d\mu (t) \leq \left(\int_{[-r,r]} \norm{x(t)}_{L^2}^2 \,d\mu (t)\right)^{\frac{1}{2}} \left(\int_{[-r,r]} 1^2 \,d\mu (t)\right)^{\frac{1}{2}} ,
\end{equation*}
then
$$\frac{1}{\mu([-r,r])}\int_{[-r,r]} \left( E\norm{x(t)}^2\right)^{\frac{1}{2}} \,d\mu (t) \leq \left(\frac{1}{\mu([-r,r])}\int_{[-r,r]}  E\norm{x(t)}^2 \,d\mu (t)\right)^{\frac{1}{2}}  ,$$
consequently if $x$ is square-mean $\mu$-ergodic, then $x$ satisfies \eqref{Eq9}.
\vskip 1 mm
Secondly $x$ is stochastically bounded: there exists $M>0$ such that
$\ds E\norm{x(t)}^2 \leq M$, then
$$E\norm{x(t)}^2 = \norm{x(t)}_{L^2}^2 \leq \norm{x(t)}_{L^2} \sup_{t\in\R}\norm{x(t)}_{L^2} \leq M^{\frac{1}{2}}\left(E\norm{x(t)}^2 \right)^{\frac{1}{2}} ,$$
therefore
\begin{equation}
\label{Eq8}
\frac{1}{\mu([-r,r])}\int_{[-r,r]} E\norm{x(t)}^2 \,d\mu (t) \leq \frac{M^{\frac{1}{2}}}{\mu([-r,r])}\int_{[-r,r]} \left( E\norm{x(t)}^2\right)^{\frac{1}{2}} \,d\mu (t) ,
\end{equation}
consequently if $x$ satisfies \eqref{Eq9}, then $x$ is square-mean $\mu$-ergodic.
\vskip 2 mm
{\bf ii)} results of i) and Remark \ref{Rq1}.
\end{proof}

\begin{remark}
A consequence of Proposition \ref{Prop1} is that all results of \cite{27}  on $\mu$-pseudo almost periodic functions
are directly applicable to square-mean $\mu$-pseudo almost periodic stochastic processes.
\end{remark}

\vskip 2 mm
In view of Proposition \ref{Prop1}, we denote by $PAP(\mathbb{R},L^2(\Omega, P, H),\mu)$ the space of all square-mean $\mu$-pseudo almost periodic stochastic processes $x : \R\to L^2(\Omega, P, H)$. We denote by $\norm{\cdot}_{\infty}$ the norm of $PAP(\mathbb{R},L^2(\Omega, P, H),\mu)$ defined by $\ds\norm{x}_{\infty}=\sup_{t\in\R}\left(E\norm{x(t)}^2\right)^{\frac{1}{2}}$.

\begin{theorem} 
Let $\mu\in\mathcal{M}$. Then $(PAP(\mathbb{R},L^2(\Omega, P, H),\mu),\left\Vert \cdot\right\Vert _{\infty})$ is a Banach space.
\end{theorem}

\begin{proof}
It is a consequence of Theorem \ref{Tw3} and Proposition \ref{Prop1}.
\end{proof}

\subsection{Almost automorphic case}

\begin{definition}
\cite{Fu-Liu}. A continuous stochastic process $x : \R\to L^2(\Omega, P, H)$ is said to be
 \textit{square-mean almost automorphic}  if for every sequence of real numbers $(s_{n}^{\prime
})_{n}$ 
there exists a subsequence of the sequence $(s_{n}^{\prime})_{n}$ denoted by $(s_{n})_{n}$ such that
for some stochastic process 
$y : \R\to L^2(\Omega, P, H)$ 
$$\forall t\in\R,\qquad  \lim_{t\to \infty} E\norm{x(t+s_{n}) - y(t)}^2=0$$
and
$$\forall t\in\R,\qquad \lim_{t\to \infty} E\norm{y(t-s_{n}) - x(t)}^2=0 .$$
\end{definition}

\begin{remark}
\label{Rq2}
It is easy to verify that a stochastic process $x : \R\to L^2(\Omega, P, H)$ is square-mean almost automorphic
if and only if
$x$ is an almost automorphic function from $\R$ to the Banach space $X=L^2(\Omega, P, H)$ in the sense of Definition \ref{Def6}, i.e. $x\in AA(\R,X)$. 
\end{remark}

\begin{definition}
\cite{Di-Ez-Mb-2}
\label{Def8}
Let $\mu\in\mathcal{M}$. A continuous and bounded stochastic process $x : \R\to L^2(\Omega, P, H)$
is said to be \textit{square-mean $\mu$-pseudo almost automorphic} if 
\[
x = y+z
\]
where $y$ is square-mean almost automorphic and $z$ square-mean $\mu$-ergodic.
\end{definition}

\vskip 2 mm
From i) of Proposition \ref{Prop1} and Remark \ref{Rq2}, we obtain the following result:

\begin{proposition}
\label{Prop2}
Let $\mu\in\mathcal{M}$. Let $x : \R\to L^2(\Omega, P, H)$ be a stochastic process.
$x$ is square-mean $\mu$-pseudo almost automorphic
if and only if
$x$ is a $\mu$-pseudo almost automorphic function from $\R$ to the Banach space $X=L^2(\Omega, P, H)$ in the sense of Definition \ref{Def7}, i.e. $x\in PAA(\mathbb{R}
,X,\mu)$.
 \end{proposition}

\begin{remark}
\label{Rq7}
A consequence of Proposition \ref{Prop2} is that all results of \cite{26}  on $\mu$-pseudo almost automorphic functions
are directly applicable to square-mean $\mu$-pseudo almost automorphic stochastic processes.
\end{remark}

\vskip 2 mm
In view of Proposition \ref{Prop2}, we denote by $PAA(\mathbb{R},L^2(\Omega, P, H),\mu)$ the space of all square-mean $\mu$-pseudo almost automorphic stochastic processes $x : \R\to L^2(\Omega, P, H)$. We denote by $\norm{\cdot}_{\infty}$ the norm of $PAA(\mathbb{R},L^2(\Omega, P, H),\mu)$ defined by $\ds\norm{x}_{\infty}=\sup_{t\in\R}\left(E\norm{x(t)}^2\right)^{\frac{1}{2}}$.

\begin{theorem} 
Let $\mu\in\mathcal{M}$. Then $(PAA(\mathbb{R},X,\mu),\left\Vert \cdot\right\Vert _{\infty})$ is a Banach space.
\end{theorem}

\begin{proof}
It is a consequence of Theorem \ref{Tw4} and Proposition \ref{Prop2}.
\end{proof}

\section{Spaces of $\mu$-pseudo almost periodic functions defined on the half-line}
\label{iv}

In this section $X$ is a Banach space. We study the case where the functions are defined only on the half-line.
\vskip2mm

\begin{definition}
$f\in C(\R^+,X)$
(continuous)
  is said to be
 \textit{almost periodic} (in the Bohr sense) if for all $\varepsilon>0$, there exists
${\ell}>0$, such that for all $\alpha \geq 0$, there exists $\tau
\in\lbrack\alpha,\alpha+{\ell}]$ with
\[
\sup_{t\geq 0}\left\|  f(t+\tau)-f(t)\right\|  <\varepsilon.
\]
We denote the space of all such functions by $AP(\mathbb{R}^+,X)$.
\end{definition}

\vskip 2mm

$BC(\mathbb{R}^+,X)$ denotes the
Banach space of all continuous and bounded functions from $\mathbb{R}^+$ to $X$,
equipped with the norm
\[
\left\|  f\right\|  _{\infty}=\sup_{t\geq 0}\left\|  f(t)\right\|  .
\]
One should point out that combining the extension theorem of an almost periodic function on $\R^+$ to an almost periodic function on $\R$ (\cite{Ar-Ba-Hi-Ne}, Proposition 4.7.1, p. 305) and the recurrence property of an almost periodic function on $\R$ which is a consequence of the definition of the almost periodicity:  there exists a sequence of real numbers $(t_n)_{n\in\N}$ such that $\lim_{n\to+\infty} t_n = +\infty$ and $\lim_{n\to+\infty} f(t+t_n) = f(t)$, one can prove that the restriction operator 
\begin{equation}
\label{Eq3}
R: AP(\mathbb{R},X) \to AP(\mathbb{R}^+,X) \quad\text{defined by }R(f)=f_{\mid_{\R^+}}
\end{equation}
is well-defined, maps $AP(\R,X)$ onto $AP(\R^+,X)$ and satisfies for $f\in AP(\R,X)$ 
\[
\norm{f}_{AP(\R,X)} = \sup_{t\in\R}\left\|  f(t)\right\|=\sup_{t\geq 0}\left\|  f(t)\right\| = \norm{R(f)}_{AP(\R^+,X)}  .
\]

\vskip 2mm

Similarly as for $\mu$-pseudo almost periodic functions defined on the whole real line, we define below the $\mu$-pseudo almost periodic functions defined on the half line.

\vskip 1mm

We denote by $\mathcal{B}_+$ the Lebesgue $\sigma$-field of $\mathbb{R}_+$ and by
$\mathcal{M}_+$ the set of all positive measures $\mu$ on $\mathcal{B}_+$
satisfying $\mu(\mathbb{R}^+)=+\infty$ and $\mu([a,b])<+\infty$, for all $0\leq a\leq b$.

\begin{definition}
\label{Def3}
Let $\mu\in\mathcal{M}_+$. $f\in BC(\R^+,X)$
is said to be \textit{$\mu$-ergodic} if 
\[
\underset{r\rightarrow+\infty}{\lim}\frac{1}{\mu([0,r])}\int_{[0,r]}
\left\Vert f(t)\right\Vert \,d\mu(t)=0 .
\]
We denote the space of all such functions by $\mathcal{E}(\mathbb{R}^+
,X,\mu)$.
\end{definition}

\begin{theorem} 
Let $\mu\in\mathcal{M}_+$. Then $(\mathcal{E}(\mathbb{R}^+
,X,\mu),\left\Vert \cdot\right\Vert _{\infty})$ is a Banach space.
\end{theorem}

\begin{proof}
It is a slight adaptation of the proof of the completeness of the space  of  $\mu$-ergodic functions defined on the whole real line: $(\mathcal{E}(\mathbb{R}
,X,\mu),\left\Vert \cdot\right\Vert _{\infty})$ (\cite{27}, Proposition 2.11).
\end{proof}

\begin{definition}
Let $\mu\in\mathcal{M}_+$. $f\in BC(\R^+,X)$ is said to be \textit{$\mu$-pseudo almost periodic} if 
\[
f = g+\phi \quad \text{ where }g\in AP(\mathbb{R}^+,X) , \quad \phi\in\mathcal{E}(\mathbb{R}^+,X,\mu).
\]
We denote the space of all such functions by $PAP(\mathbb{R}^+,X,\mu)$. 
\end{definition}

\begin{theorem} 
\label{Tw5}
Let $\mu\in\mathcal{M}_+$. Then $(PAP(\mathbb{R}^+,X,\mu),\left\Vert \cdot\right\Vert _{\infty})$ is a Banach space.
\end{theorem}

\begin{proof}
We use Theorem \ref{Tw2} by setting $U=\R^+$, $\mathcal{F}_1=AP(\R^+,X)$ and $\mathcal{F}_2=\mathcal{E}(\mathbb{R}^+,X,\mu)$.
\vskip 2 mm
{\bf i)} 
$AP(\R^+,X)$ is isometrically isomorphic to the Banach space $AP(\R,X)$, then it is a Banach space. 
Let $\Phi\in C(X,X)$ and $g\in AP(\R^+,X)$. First there exists $f\in AP(\R,X)$ such that $R(f)=g$ (see \eqref{Eq3}). Secondly by using the definition of the restriction operator $R$ 
we obtain $\Phi \circ g = \Phi \circ [R(f)] = R(\Phi \circ f)$. From the fact that $\Phi\circ f \in AP(\R,X)$ (see proof of Theorem \ref{Tw3}) we deduce that $\Phi\circ g \in AP(\R^+,X)$; so Hypothesis i) is satisfied.
\vskip 2 mm
{\bf ii)} By using Definition \ref{Def3}, we deduce that the pair $AP(\R^+,X)$ and
$\mathcal{E}(\mathbb{R}^+,X,\mu)$ satisfies Hypothesis ii). 
\vskip 2 mm
By Theorem \ref{Tw2}, we obtain that $PAP(\mathbb{R}^+,X,\mu)=AP(\mathbb{R}^+,X) + \mathcal{E}(\mathbb{R}^+,X,\mu)$  is a Banach space endowed with the norm $\norm{\cdot}_{\infty}$. 
\end{proof}

The following definition is due to Maurice Fr\'echet \cite{Fr, Fr1}.

\begin{definition} 
$f\in BC(\R^+,X)$ is said to be \textit{asymptotically almost periodic} if 
\[
f = g+\phi \quad \text{ where }g\in AP(\mathbb{R}^+,X) , \quad  \phi\in C(\R^+,X)  \text{ and } \lim_{t\to + \infty}\phi(t)=0 .
\]
We denote the space of all such functions by $AAP(\mathbb{R}^+,X)$. 
\end{definition}

We will see that the following well-known result is a simple corollary of our main result.

\begin{theorem} 
$(AAP(\mathbb{R}^+,X),\left\Vert \cdot\right\Vert _{\infty})$ is a Banach space.
\end{theorem}

\begin{proof}
We use Theorem \ref{Tw2} by setting $U=\R^+$, $\mathcal{F}_1=AP(\R^+,X)$ and $\mathcal{F}_2=C_0(\R^+,X)$ where $C_0(\R^+,X) = \set{\phi\in C(\R^+,X)}{\lim_{t\to + \infty}\phi(t)=0}$
\vskip 2 mm
{\bf i)} is proved in i) of the proof of Theorem \ref{Tw5}.
\vskip 2 mm
{\bf ii)} It is obvious. 
\vskip 2 mm
By Theorem \ref{Tw2}, we deduce that $AAP(\mathbb{R}^+,X)=AP(\mathbb{R}^+,X) + C_0(\R^+,X)$  is a Banach space endowed with the norm $\norm{\cdot}_{\infty}$. Note that this sum is direct: $AAP(\mathbb{R}^+,X)=AP(\mathbb{R}^+,X) \bigoplus C_0(\R^+,X)$.

\end{proof}

\section{Spaces of weighted pseudo almost periodic and automorphic sequences}
\label{iii}

In this Section we provide results of completeness on the space of weighted pseudo almost periodic sequences and of the space of weighted pseudo automorphic sequences without any assumption on the weight of the ergodic part of these sequences.

\subsection{Almost periodic case}

The notion of weighted pseudo almost periodicity of sequences 
is a generalization of pseudo almost periodicity of sequences (\cite{Zha2}, Definition 7.11, p. 84). 
The notion of weighted pseudo almost periodicity of sequences is due Zhang and Li \cite{Zh-Li}.

\vskip 2 mm
Denote by $X^{\Z}$ the set of all two-sided sequences  $u=(u_n)_{n\in\Z}$
with values in the Banach space $X$. $\ell^{\infty}(\mathbb{Z},X)$ denotes the
Banach space of all bounded two-sided sequences of $X^{\Z}$
equipped with the norm
\[
\left\|  u\right\|  _{\infty}=\sup_{n\in\mathbb{Z}}\left\|  u_n \right\|  .
\]

\begin{definition}
(\cite{Co}, p. 45).
$u=(u_n)_{n\in\Z}\in X^{\Z}$  is said to be
\textit{almost periodic} if for all $\varepsilon>0$, there exists
$N\in\N^*$, such that for all $m\in\Z$, there exists $p \in\{m, m+1, \cdots, m+N \}$ with
\[
\sup_{n\in\Z}\left\|  u_{n+p}-u_{n}\right\|  <\varepsilon.
\]
We denote the space of all such sequences by $AP_d(\Z,X)$.
\end{definition}

\vskip 2mm

The following well-known statement (\cite{Co}, Theorem 1.27, p. 47) clarifies the relation between almost periodic sequences and functions:  a necessary and sufficient condition for a sequence $u=(u_n)_{n\in\Z}\in AP_d(\Z,X)$ is the existence of $f\in AP(\R,X)$ such that $u_n=f(n)$ for all $n\in\Z$. This permits us to assert that the restriction operator 
\begin{equation}
\label{Eq5}
R: AP(\mathbb{R},X) \to AP_d(\Z,X) \quad\text{defined by }R(f)=f_{\mid_{\Z}}
\end{equation}
is well-defined and maps $AP(\R,X)$ onto $AP_d(\Z,X)$. Let $Q=AP(\R,X)/{\rm Ker} \, R$ and let $\Pi_0: AP(\R,X)\to Q$ be the quotient map, so the surjective linear map $R$ induces  a unique isomorphism 
\begin{equation*}
\tilde{R}: Q \to AP_d(\Z,X) \quad\text{defined by } R=\tilde{R}\circ \Pi_0
\end{equation*}
satisfying for $f\in AP(\R,X)$
\[
\norm{\Pi_0(f)}_{Q} = \inf\set{\sup_{t\in\R}\norm{f(t)+g(t)}}{g\in {\rm Ker} \, R}    = \sup_{n\in\Z}\norm{f(n)} = \norm{\tilde{R}(\Pi_0(f))}_{AP_d(\Z,X)}  ,
\]
therefore $Q=AP(\R,X)/{\rm Ker} \, R$ is isometrically isomorphic to $AP_d(\Z,X)$.

\vskip 2mm

We denote by by
$\mathcal{M}_d$ the set of all two-sided sequences $p=(p_n)_{n\in\Z}\in \R^{\Z}$
satisfying $p_n\geq0$, for all $n\in\Z$ and
$\ds\sum_{n\in\Z}p_n=+\infty$.

\begin{definition}
\label{Def2}
\cite{Ab, Zh-Li}
 Let $p=(p_n)_{n\in\Z}\in\mathcal{M}_d$. $u=(u_n)_{n\in\Z}\in \ell^{\infty}(\mathbb{Z},X)$
is said to be \textit{weighted ergodic} if 
\[
\underset{N\rightarrow+\infty}{\lim}\frac{1}{\sum_{n=-N}^{N}p_n}\sum_{n=-N}^{N} 
\left\Vert u_n\right\Vert p_n =0 .
\]
We denote the space of all such functions by $\mathcal{E}_d(\mathbb{Z},X,p)$.
\end{definition}

\begin{proposition}
Let $p\in\mathcal{M}_d$. Then $(\mathcal{E}_d(\Z
,X,p),\left\Vert \cdot\right\Vert _{\infty})$ is a Banach space.
\end{proposition}

\begin{proof}
It is enough to prove that $\mathcal{E}_d(\mathbb{Z},X,p)$ is closed in
$\ell^{\infty}(\mathbb{Z},X)$. Let $(u^{n})_{n}$ be a sequence in $\mathcal{E}_d
(\mathbb{Z},X,p)$ such that $\ds\lim_{n\rightarrow+\infty}u^{n}=u$ (cv in $\ell^{\infty}(\mathbb{Z},X)$). Then we have
\[
 \sum_{m=-N}^{N} \norm{u_m} p_m \leq \sum_{m=-N}^{N} \norm{u_m - u_m^n} p_m 
 + \sum_{m=-N}^{N} \norm{u_m^n} p_m
\]
we deduce that
\[
 \frac{1}{\sum_{n=-N}^{N}p_n}\sum_{m=-N}^{N} \norm{u_m} p_m \leq \sup_{m\in\Z}\norm{u_m - u_m^n} + \frac{1}{\sum_{n=-N}^{N}p_n}\sum_{m=-N}^{N} \norm{u_m^n} p_m ,
\]
it follows 
\[
\limsup_{N\to+\infty} \frac{1}{\sum_{n=-N}^{N}p_n}\sum_{m=-N}^{N} \norm{u_m} p_m \leq \sup_{m\in\Z}\norm{u_m - u_m^n} \quad\text{ for all }n\in\mathbb{N}.
  \]
Since $\displaystyle   \lim_{n\to+\infty} \norm{u - u^n}_{\ell^{\infty}(\mathbb{Z},X)} = \lim_{n\to+\infty}\left(\sup_{m\in\Z}\norm{u_m - u_m^n}\right)=0$, we deduce that
\[
\lim_{N\to+\infty} \frac{1}{\sum_{n=-N}^{N}p_n}\sum_{m=-N}^{N} \norm{u_m} p_m=0.
\]
\end{proof}

\begin{definition}
\cite{Zh-Li}
Let $p=(p_n)_{n\in\Z}\in\mathcal{M}_d$. $u=(u_n)_{n\in\Z}\in \ell^{\infty}(\mathbb{Z},X)$ is said to be \textit{weighted pseudo almost periodic} if 
\[
u = v+w \quad \text{ where }v\in AP_d(\mathbb{Z},X) , \quad w\in\mathcal{E}_d(\mathbb{Z},X,p).
\]
We denote the space of all such functions by $PAP_d(\mathbb{Z},X,p)$. 
\end{definition}

\begin{remark} 
Contrarily to the pseudo almost periodic case: $p_n=1$ for all $n\in\Z$, in general  the sum is not direct for the weighted pseudo almost periodic case. For example if $p=(p_n)_{n\in\Z}$ is defined by $p_{2n}=0$, $p_{2n+1}=1$ and $u=(u_n)_{n\in\Z}$ by $u_{2n}=1$, $u_{2n+1}=0$ for all $n\in\Z$, then $u\in AP_d(\mathbb{Z},\R)\cap\mathcal{E}_d(\mathbb{Z},\R,p)$.
\end{remark} 

\begin{theorem} 
\label{Tw6}
Let $p\in\mathcal{M}_d$. Then $(PAP_d(\mathbb{Z},X,p),\left\Vert \cdot\right\Vert _{\infty})$ is a Banach space.
\end{theorem}

\begin{proof}
We use Theorem \ref{Tw2} by setting $U=\Z$, $\mathcal{F}_1=AP_d(\Z,X)$ and $\mathcal{F}_2=\mathcal{E}_d(\mathbb{Z},X,p)$. Remark that the Banach space $B(\Z,X)$ used in Theorem \ref{Tw2} is $\ell^{\infty}(\mathbb{Z},X)$.
\vskip 2 mm
{\bf i)} $AP_d(\Z,X)$ is isometrically isomorphic to the Banach space $AP(\R,X)/{\rm Ker} \, R$, then it is a Banach space. 
Let $\Phi\in C(X,X)$ and $u = (u_n)_{n\in\Z}\in AP_d(\Z,X)$. 
First there exists $f\in AP(\R,X)$ such that $R(f)=u$ (see \eqref{Eq5}). 
Secondly by using the definition of the restriction operator $R$ we obtain $\Phi \circ u = \Phi \circ [R(f)] = R(\Phi \circ f)$. 
From the fact that $\Phi\circ f \in AP(\R,X)$ (see proof of Theorem \ref{Tw3}) we deduce that $\Phi\circ u \in AP_d(\Z,X)$; so Hypothesis i) is satisfied.
\vskip 2 mm
{\bf ii)} By using Definition \ref{Def2}, we deduce that the pair $AP_d(\Z,X)$ and
$\mathcal{E}_d(\mathbb{Z},X,p)$ satisfies Hypothesis ii).
\vskip 2 mm
By Theorem \ref{Tw2}, we obtain that $PAP_d(\mathbb{Z},X,p)=AP_d(\mathbb{Z},X) + \mathcal{E}_d(\mathbb{Z},X,p)$  is a Banach space endowed with the norm $\norm{\cdot}_{\infty}$. 
\end{proof}

\subsection{Almost automorphic case}

\begin{definition}
\cite{Ara}.
 $u=(u_n)_{n\in\Z}\in X^{\Z}$  is said to be \textit{almost automorphic} if for every sequence of integer numbers $(k_{n}^{\prime
})_{n\in\N}\in \Z^\N$ 
there exists a subsequence of $(s_{n}^{\prime})_{n}$ denoted by $(s_{n})_{n}$ such that

\[
v_p=\underset{n\rightarrow\infty}{\lim}u_{p+k_{n}}\text{ exists for all
}p\text{ in }\mathbb{Z}
\]
and
\[
\underset{n\rightarrow\infty}{\lim}v_{p-k_{n}}=u_p\text{ for all }p\text{ in
}\mathbb{Z}\text{.}
\]
We denote the space of all such functions by $AA_d(\mathbb{Z},X)$.

\end{definition}

The following definition is due to Abbas  in  \cite{Ab} but it is not proved that this space is closed in $\ell^{\infty}(\mathbb{Z},X)$.

\begin{definition}
Let $p=(p_n)_{n\in\Z}\in\mathcal{M}_d$. $u=(u_n)_{n\in\Z}\in \ell^{\infty}(\mathbb{Z},X)$ is said to be \textit{weighted pseudo almost automorphic} if 
\[
u = v+w \quad \text{ where }v\in AA_d(\mathbb{Z},X) , \quad w\in\mathcal{E}_d(\mathbb{Z},X,p).
\]
We denote the space of all such functions by $PAA_d(\mathbb{Z},X,p)$. 
\end{definition}

\begin{theorem} 
Let $p\in\mathcal{M}_d$. Then $(PAA_d(\mathbb{Z},X,p),\left\Vert \cdot\right\Vert _{\infty})$ is a Banach space.
\end{theorem}

\begin{proof}
We use Theorem \ref{Tw2} by setting $U=\Z$, $\mathcal{F}_1=AA_d(\Z,X)$ and $\mathcal{F}_2=\mathcal{E}_d(\mathbb{Z},X,p)$. \vskip 2 mm
{\bf i)} $AA_d(\Z,X)$ is a closed vector subspace of $\ell^{\infty}(\mathbb{Z},X)$. If $\Phi\in C(X,X)$ and $v\in AA_d(\Z,X)$ then $\Phi\circ v\in AA(\Z,X)$ \cite{Ara}; so Hypothesis i) is satisfied.
\vskip 2 mm
{\bf ii)}  is similar to ii) in the proof of Theorem \ref{Tw6}.
\vskip 2 mm
By Theorem \ref{Tw2}, we obtain that $PAA_d(\mathbb{Z},X,p)=AA_d(\mathbb{Z},X) + \mathcal{E}_d(\mathbb{Z},X,p)$  is a Banach space endowed with the norm $\norm{\cdot}_{\infty}$. 
\end{proof}


\end{document}